\documentclass[11pt,reqno]{amsart}
\usepackage[foot]{amsaddr}
\usepackage[utf8]{inputenc}
\usepackage{amsmath}
\usepackage{amsfonts}
\usepackage{amssymb}
\usepackage{amsthm}
\usepackage{geometry}
\usepackage{enumitem}
\usepackage[usenames,dvipsnames,svgnames]{xcolor}
\usepackage{hyperref}
\usepackage[utf8]{inputenc}
\usepackage[english]{babel}
\usepackage[symbol]{footmisc}
\usepackage[centercolon]{mathtools}
\usepackage{microtype}
\usepackage{lmodern}
\usepackage[only,llbracket,rrbracket]{stmaryrd}
\usepackage{tikz}
\usepackage{dsfont}

\hypersetup{
    colorlinks=true,
    linkcolor=blue,     
    urlcolor=blue,
    citecolor=gray
}

\DeclareMathAlphabet{\mathup}{OT1}{\familydefault}{m}{n}
\newcommand{\dd}[1]{\mathop{}\!\mathup{d} #1}

\def\XXint#1#2#3{{\setbox0=\hbox{$#1{#2#3}{\int}$ }
\vcenter{\hbox{$#2#3$ }}\kern-.59\wd0}}

\newtheorem{theorem}{Theorem}[section]
\newtheorem{lemma}{Lemma}[section]

\theoremstyle{definition}

\theoremstyle{remark}
\newtheorem{remark}{Remark}[section]

\numberwithin{equation}{section}

\usepackage{mathrsfs}

\begin{document}

\title[Enhanced Dissipation via the Malliavin Calculus]{Enhanced Dissipation via the Malliavin Calculus}

\author[D. Villringer]{David Villringer}
\address{Department of Mathematics, Imperial College London, London, SW7 2AZ, UK}
\email{d.villringer22@imperial.ac.uk}



\begin{abstract}
In this work we investigate the phenomenon of enhanced dissipation using techniques from the Malliavin Calculus. In particular, we construct a concise, elementary argument, that allows us to recover the well-known enhanced dissipation timescale for shear flows, first obtained in \cite{Bedrossian_CotiZelati_2017}, as well as the precise hypoelliptic regularisation in Fourier space.
\end{abstract}

\maketitle

\tableofcontents
\section{Introduction and main results}
On the two-dimensional periodic torus $\mathbb{T}^2$, we consider passive scalars advected by a shear flow and satisfying the hypoelliptic drift-diffusion equation 
\begin{equation}
\label{eq: hypoelliptic equation}
\begin{cases}
\partial_t f+u(y)\partial_x f=\frac{1}{2}\nu \partial^2_y f, \\
f(0)=f_0,
\end{cases}
\end{equation}
where $u=u(y):\mathbb{T}\to \mathbb{R}$ is a fixed, smooth function, $\nu\in (0,1]$ is a diffusivity parameter, and $f_0:\mathbb{T}^2\to \mathbb{R}$ is an $L^2(\mathbb{T}^2)$-initial datum such that
\begin{equation}\label{eq:meanzero}
\int_\mathbb{T}f_0(x,y)\dd x =0, \qquad \forall y\in \mathbb{T},
\end{equation}
which is a property preserved by the equation.


Throughout the article, we will assume that $u$ has a finite number of critical points, of maximal order $n_0$ (i.e., for any $y \in \mathbb{T}$, there exists $n \leq n_0$ so that $u^{(n+1)}(y)\neq 0$). The above equation \eqref{eq: hypoelliptic equation} has received a great deal of attention in recent years. In particular, much has been said about the particular interplay between the transport term and the diffusivity, which in principle should allow for faster convergence to equilibrium than the natural dissipative rate proportional to $\nu$ -- a phenomenon known as \emph{enhanced dissipation}. Indeed, in \cite{Constantin_Kiselev_Ryzhik_Zlatos_2008} the authors derived a spectral criterion for when solutions to a given abstract PDE in a Hilbert space experience enhanced dissipation. (This provides one way of seeing where the condition \eqref{eq:meanzero} comes from). More recently, in \cite{Bedrossian_CotiZelati_2017} the authors studied the equation \eqref{eq: hypoelliptic equation}, and were able to show, using techniques from hypocoercivity, that there holds the \emph{enhanced dissipation estimate} 
\begin{equation}\label{eq:logstuff}
\|S_\nu(t)P_k\|_{L^2\to L^2} \leq C e^{-\epsilon\frac{\nu^{\frac{n_0+1}{n_0+3}}|k|^\frac{2}{n_0+3}t}{(1+|\log(\nu)|+|\log(|k|)|)^2}}
\end{equation}
where $S_\nu(t)$ denotes the semigroup generated by \eqref{eq: hypoelliptic equation}, and $P_k$ denotes the projection onto the $k^{th}$ Fourier mode in $x$. The logarithmic factor was subsequently removed in \cite{Albritton_Beekie_Novack_2022a}, employing methods from hypoellipticity, and in \cite{CZG23}, via pseudospectral estimates inspired by \cite{Wei_2019}.

The present article is is most spiritually in line with \cite{Albritton_Beekie_Novack_2022a}.
Indeed, the development of the Malliavin calculus was in large part motivated by the desire to formulate an elementary probabilistic proof of H\"ormander's theorem \cite{Hormander_2009}, which avoids having to work with pseudodifferential operators. Inspired by this, we show that employing only basic tools from Stochastic Analysis, the Malliavin Calculus provides a tool for proving the Semigroup estimate \eqref{eq:logstuff} in a concise way. Indeed, we will be able to recover the following result:
\begin{theorem}[\cite{Bedrossian_CotiZelati_2017, Albritton_Beekie_Novack_2022a, CZG23}]
\label{main result}
Let $P_k$ denote the projection onto the $k^{th}$ fourier mode in x, and let $S_\nu$ denote the Semigroup generated by \eqref{eq: hypoelliptic equation}. Then, there exist constants $C,c,\Tilde{\nu}>0$ independent of $f_0$, so that for any $\nu |k|^{-1}\leq \Tilde{\nu}$ it holds
\begin{equation}
\label{eq: main result}
\|S_\nu(t)P_k\|_{L^2\to L^2} \leq C e^{-c|k|^{\frac{2}{n_0+3}}\nu^{\frac{n_0+1}{n_0+3}}t},
\end{equation}
for any $t\geq0$.
\end{theorem}
\begin{remark}
Although from our viewpoint the principle interest of this estimate is its relationship to enhanced dissipation, we note that it also provides estimates on the degree of hypoelliptic regularisation that the equation enjoys. Indeed, as noted already in \cite{Bedrossian_CotiZelati_2017}, standard parabolic regularisation in $y$ together with \eqref{eq: main result} shows that the solution to \eqref{eq: hypoelliptic equation} lives in the Gevrey space 
$$
\mathcal{G}^\frac{n_0+3}{2}:=\bigcup_{\lambda>0}\{f \in L^2: e^{\lambda |\nabla|^\frac{2}{n_0+3}}f \in L^2\}
$$ 
for any positive $t$.
\end{remark}
The proof of Theorem \ref{main result} will follow by applying the Malliavin integration by parts formula to functions of the stochastic characteristics associated to \eqref{eq: hypoelliptic equation}, given by 
\begin{equation}
\dd X_t =\begin{pmatrix}
u(X_t^{2})\\
0
\end{pmatrix} \dd t+\begin{pmatrix}
0\\
\sqrt{\nu}\dd B_t
\end{pmatrix}
\end{equation}
where $B_t$ is a standard Brownian motion.  The exact construction will be detailed in the next section. In particular, due to the simplicity of this process, we will be able to perform explicit computations on quantites related to the Malliavin matrix, allowing us to extract the optimal dissipation timescales.
\subsection{Malliavin Calculus Preliminaries}
We begin by recalling a couple of elementary facts and constructions from the Malliavin calculus. A more thorough introduction can be found in \cite{Nualart_1995a} (see also the excellent set of notes \cite{Hairer}). Set $H=L^2(\mathbb{R}_+)$, and let $(\Omega,\mathcal{F},\mathbb{P})$ be a probability space. An isonormal Gaussian process $W$ is  an isometry $W:H\to L^2(\Omega,\mathbb{P})$, i.e. a map satisfying 
\begin{equation}
\mathbb{E}(W(h)W(g))=\langle h,g\rangle_{H}, \qquad \forall h,g\in H.
\end{equation}
In particular, an isonormal Gaussian process defines a standard Brownian motion by setting 
\begin{equation}
B_t=W(\mathds{1}_{[0,t]}).
\end{equation}
The Malliavin derivative is then an unbounded, closed operator $\mathcal{D}:L^2(\Omega,\mathbb{P}) \to L^2(\Omega,\mathbb{P},H)$ where we denote its domain by $\mathbb{D}^{1,2}:=D(\mathcal{D})$, which acts on functions of the form $G=F(W(h_1),\dots W(h_n))$, where $F$ is smooth and has at most polynomial growth at infinity, by 
\begin{equation}
\mathcal{D}G=\sum_{i=1}^n \partial_i F(W(h_1),\dots W(h_n))h_i .
\end{equation}
More generally, we can define the Malliavin derivative as an operator from $L^2(\Omega,\mathbb{P},K) \to L^2(\Omega,\mathbb{P},H\otimes K)$, where $K$ is some Hilbert space, in which case we denote its domain by $\mathbb{D}^{1,2}(K)$.
We denote the adjoint of the Malliavin derivative by $\delta$, which is known as the Skorokhod integral. An important fact about the Skorokhod integral is that for any $h \in \mathbb{D}^{1,2}(H)$, there holds 
\begin{equation}
\label{Skorokhod bound}
\mathbb{E}(\delta h)^2 =\mathbb{E}\int_0^\infty h_s^2 \dd s+\mathbb{E}\int_0^\infty \int_0^\infty \mathcal{D}_s h_t \mathcal{D}_t h_s \dd s \dd t.
\end{equation}
Given a Malliavin differentiable random variable $X \in \mathbb{R}^d$, we can associate the Malliavin matrix $\mathcal{M}$, which is a symmetric, positive semi-definite matrix with components given by 
\begin{equation}
\mathcal{M}_{i,j}=\langle \mathcal{D}X_i,\mathcal{D}X_j \rangle.
\end{equation}
In particular, under the additional assumption that the Malliavin matrix is almost surely invertible and has inverse moments of all orders, we obtain the following ``integration by parts" formula: Let $X \in \mathbb{D}^{1,2}$. If $G:\mathbb{R}^d \to \mathbb{R}$ is a smooth function, there holds
\begin{equation}
\label{integration by parts}
\mathbb{E}(\partial_i G(X))=\mathbb{E}(G(X)\delta(Y_i)),
\end{equation}
with $Y_i=(\mathcal{D}X)^*\mathcal{M}^{-1}e_i$,
where $e_i$ is the $i^{th}$ euclidean basis vector.

\subsection{Main ideas of the proof}
As noted above, one can associate the following ``stochastic characteristics" to our PDE  \eqref{eq: hypoelliptic equation}.
\begin{align}
X^1_t(x,y)=x+\int_{0}^tu\left(y+\sqrt{\nu} B_s\right) \dd s, \qquad X^2_t(x,y)=y+\sqrt{\nu} B_t.
\end{align}
Indeed, denoting the vector $X_t=(X_t^1,X_t^2)$ with notation as above, an application of Ito's formula shows that the solution to \eqref{eq: hypoelliptic equation} is then given by $f(t,x,y)=\mathbb{E}f_0(X_t^{-1}(x,y))$.
Elementary calculations show that the Malliavin derivative is given by
\begin{equation}
\mathcal{D}_rX_t = \begin{pmatrix}
\displaystyle\sqrt{\nu} \int_r^t u'(y+\sqrt{\nu}B_s)\dd s \vspace{2mm}\\
\sqrt{\nu} \mathds{1}_{r\leq t},
\end{pmatrix}
\end{equation}
and therefore, the Malliavin matrix is just
\begin{equation}
\mathcal{M} =\nu \begin{pmatrix}
\displaystyle\int_{0}^t \left(\int_r^t u'\left(y+\sqrt{\nu}B_s\right)\dd s\right)^2\dd r & \displaystyle\int_{0}^t \int_r^t u'\left(y+\sqrt{\nu}B_s\right)\dd s \dd r\vspace{2mm}\\
\displaystyle\int_{0}^t \int_r^t u'\left(y+\sqrt{\nu}B_s\right)\dd s \dd r & t
\end{pmatrix}.
\end{equation}
The point will be then to apply the Malliavin integration by parts formula \eqref{integration by parts} to a well chosen function, which in turn entails estimaing the behaviour of $\mathcal{M}^{-1}$. At this point it may be tempting to simply go through the proof of H\"ormander's theorem and keep track of the powers of $t, \nu$ appearing. However, this approach is doomed to fail for a couple of reasons. Firstly, since the standard proof of H\"ormander's theorem relies on an application of Norris' lemma, which does not come with sharp estimates on the exponents its produces, this would make recovering the correct scaling quite challenging. More fundamentally, estimating inverse moments of the operator norm of the Malliavin matrix cannot possibly yield the decay rate we are after, since it will in particular be dominated by the component of the matrix corresponding to a ``nudge" in the $y$-direction. Since any function solely depending on $y$ cannot experience dissipation on any time scale faster than that of the heat equation, this ensures that estimating the inverse moments of the Malliavin matrix would at best leave us with a decay rate proportional to $e^{-\nu t}$. We therefore need to tailor our arguments to ensure we only capture the variations in the $x$ direction.
Fortunately, since we assume in \eqref{eq:meanzero} that $f_0$ has zero mean in $x$, we note that it may be written as $f_0=\partial_x F_0$, for some function $F_0$. Therefore, applying the expression \eqref{integration by parts}, we see that 
\begin{equation}
\mathbb{E}(f_0(X_t))=\mathbb{E}(\partial_x F_0(X_t)) =\mathbb{E}(F_0(X_t) \delta (Y_t(y))),
\end{equation}
where $\delta$ denotes the Skorokhod integral, and 
\begin{equation}
Y_t(y)=(\mathcal{D}X_t)^* \mathcal{M}^{-1}\begin{pmatrix}
    1\\
    0
\end{pmatrix},
\end{equation}
which encodes the fact that we only care about variations in the $x$-direction.
In particular, calculating $Y$ explicitly yields 
\begin{equation}
\label{Y expression}
Y_{r,t}(y)=\frac{\displaystyle\int_{r}^tu'(y+\sqrt{\nu}B_s) \dd s -\frac{1}{t}\int_{0}^t\int_m^t u'(y+\sqrt{\nu}B_s)\dd s\dd m}{\displaystyle\sqrt{\nu} \int_0^t\left(\int_{r}^tu'(y+\sqrt{\nu}B_s) \dd s -\frac{1}{t}\int_{0}^t\int_m^t u'(y+\sqrt{\nu}B_s)\dd s\dd m\right)^2\dd r}.
\end{equation}
Pretending for a second that $Y$ was adapted to the filtration generated by $(B_t)_{t \geq 0}$, the Ito isometry shows that 
\begin{equation}
\label{ito isometry}
\mathbb{E}(\delta Y_t(y))^2 \leq \mathbb{E}\left[\frac{1}{\displaystyle\nu\int_0^t\left(\int_{r}^tu'(y+\sqrt{\nu}B_s) \dd s -\frac{1}{t}\int_{0}^t\int_m^t u'(y+\sqrt{\nu}B_s)\dd s\dd m\right)^2\dd r}\right].
\end{equation}
It therefore is clear that we need to estimate the inverse moments of the following quantity, where we make a slight abuse of notation and denote it by $\det(M)$ (it actually differs from the Malliavin determinant by a factor of $\nu$):
\begin{equation}
\det(M_t)(y):=\int_0^t\left(\int_{r}^tu'(y+\sqrt{\nu}B_s) \dd s -\frac{1}{t}\int_{0}^t\int_m^t u'(y+\sqrt{\nu}B_s)\dd s\dd m\right)^2\dd r,
\end{equation}
We will occasionally suppress the dependence on $y$, whenever there is no ambiguity. At this point, we will be able to argue in much the same way as in the proof of H\"ormander's theorem. In particular, since our situation is much simpler than the general case of H\"ormander, it will be possible to proceed entirely without employing Norris' lemma. Instead, we will simply localise around critical points in an appropriate way, and employ Taylor approximation arguments, together with elementary Brownian scaling.

\section{Proof of the Main Result}
We now show how Theorem \ref{main result} may be deduced, once one has obtained good bounds on the behaviour of $\det(M)$. This way, all that will be left to do is to obtain said bounds, which will be the subject of section \ref{Malliavin Proofs}. To ease in notation, we will from now on write $\|\cdot\|_{L^\infty_t}$ to mean $\|\cdot\|_{L^\infty(0,t)}$, and similarly for H\"older norms. In the cases where spatial norms are used, we will make this explicit. Furthermore, any constants appearing are assumed positive, unless stated otherwise.

\subsection{Proof of Theorem \ref{main result}}
The main ingredient of the proof will be the following set of inequalities, which will be the result of a series of lemmas, whose proofs are left for later:
\begin{lemma}
\label{Important bounds}
Let the maximal order of critical points of $u$ be $n_0$. Then, for any $p\geq 1$ there exists a constant $C(p)>0$ independent of $\nu, t \in [0,\nu^{-1}]$ so that there holds 
\begin{align}
\sup_{y \in \mathbb{T}}\mathbb{E}(\det(M_t)(y)^{-p}) &\leq C(p)t^{-p(n_0+3)}\nu^{-n_0p},\\
\sup_{y \in \mathbb{T}}\mathbb{E}\left(\left(\frac{t\|u''(y+\sqrt{\nu}B_s)\|_{L^\infty_t}}{\det(M_t)(y)}\right)^{p}\right)&\leq C(p)t^{-p\frac{n_0+5}{2}}\nu^{-p\frac{n_0+1}{2}}.
\end{align}
\end{lemma}
We now provide a proof of Theorem \ref{main result}, assuming that Lemma \ref{Important bounds} is proved:
Firstly, we note that we have the following bound on the Skorokhod integral of $Y$:
\begin{lemma}
\label{Y bound}
For $Y$ given by \eqref{Y expression}, there exists a positive constant $C$ depending only on $u$, so that for any $t \leq \nu^{-1}$, $\nu>0$ it holds
\begin{equation}
\sup_{y \in \mathbb{T}}\mathbb{E}(\delta Y_t(y))^2 \leq \frac{C}{t^{n_0+3}\nu^{n_0+1}}.
\end{equation}
\end{lemma}
\begin{proof}
As discussed in \eqref{ito isometry}, the first term in \eqref{Skorokhod bound} is simply 
\begin{equation}
\mathbb{E}\int_0^t \frac{1}{\nu (\det(M_t)(y))^2} \left(\int_{r}^tu'(y+\sqrt{\nu}B_s) \dd s -\frac{1}{t}\int_{0}^t\int_m^t u'(y+\sqrt{\nu}B_s)\dd s\dd m\right)^2 \dd r
\end{equation}
which we recognise as 
\begin{equation}
\mathbb{E}\left(\frac{1}{\nu \det(M_t)(y)}\right) \leq C\nu^{-1-n_0}t^{-n_0-3}
\end{equation}
by Lemma \ref{Important bounds}.
Therefore, it only remains to deal with the term involving the Malliavin derivatives. Let us compute 
\begin{align}
\mathcal{D}_zY_r 
&=\frac{1}{\det(M_t)(y)}\left(\int_{z \vee r}^tu''(y+\sqrt{\nu}B_s)\dd s-\frac{1}{t}\int_0^t \int_{m \vee z}^t u''(y+\sqrt{\nu}B_s) \dd s\dd m\right)\notag\\
&\quad-\frac{g(r)}{\det(M_t)(y)^2}\left(\int_0^t\left[g(r)\int_{z \vee r}^tu''(y+\sqrt{\nu}B_s)\dd s-\frac{1}{t}\int_0^t \int_{m \vee z}^tu''(y+\sqrt{\nu}B_s) \dd s\dd m\right]\dd r\right)
\end{align}
where 
\begin{equation}
g(r)=\int_{r}^tu'(y+\sqrt{\nu}B_s) \dd s -\frac{1}{t}\int_{0}^t\int_m^t u'(y+\sqrt{\nu}B_s)\dd s\dd m.
\end{equation}
Hence, using the fact that 
\begin{align}
\left|\int_{z \vee r}^tu''(y+\sqrt{\nu}B_s)\dd s-\frac{1}{t}\int_0^t \int_{m \vee z}^tu''(y+\sqrt{\nu}B_s) \dd s\dd m\right| \leq t \|u''(y+\sqrt{\nu}B_s)\|_{L^\infty_t},
\end{align}
we may upper bound the Malliavin derivative 
\begin{equation}
|\mathcal{D}_zY_r |\leq C \frac{t\|u''(y+\sqrt{\nu}B_s)\|_{L^\infty_t}}{\det(M_t)(y)}+\frac{t\|u''(y+\sqrt{\nu}B_s)\|_{L^\infty_t}|g(r)|}{\det(M_t)(y)^2}\int_0^t  |g(m)|\dd m,
\end{equation}
where $C$ is some numerical constant independent of $u,\nu,t$.
Hence, we may estimate 
\begin{align}
&\int_0^t\int_0^t \mathcal{D}_r Y_z \mathcal{D}_z Y_r \dd r \dd z \leq C \frac{t^2 (t\|u''(y+\sqrt{\nu}B_s)\|_{L^\infty_t})^2}{\det(M_t)(y)^2}\notag\\
&\qquad+\frac{(t\|u''(y+\sqrt{\nu}B_s)\|_{L^\infty_t})^2}{\det(M_t)(y)^3}\int_0^t \int_0^t \left(\int_0^t |g(m)|\dd m\right)|g(r)|\dd r \dd z\notag\\
&\qquad+\frac{(t\|u''(y+\sqrt{\nu}B_s)\|_{L^\infty_t})^2}{\det(M_t)(y)^4}\int_0^t \int_0^t\int_0^t |g(r)| |g(z)| \left(\int_0^t|g(m)|\dd m\right)^2\dd r\dd z.
\end{align}
Note further that by Cauchy-Schwarz, $(\int_0^t|g(m)|\dd m)^2\leq t\det(M_t)(y)$. Therefore, we bound this by
\begin{equation}
\int_0^t\int_0^t \mathcal{D}_r Y_z \mathcal{D}_z Y_r \dd r \dd z\leq C \frac{t^2(t\|u''(y+\sqrt{\nu}B_s)\|_{L^\infty_t})^2}{\det(M_t(y))^2}.
\end{equation}
Hence, taking expectation, Lemma \ref{Important bounds} yields an upper bound
\begin{equation}
\mathbb{E}\left(\int_0^t\int_0^t \mathcal{D}_r Y_z \mathcal{D}_z Y_r \dd r \dd z\right)\leq C(u) t^{2}\nu^{-(n_0+1)}t^{-(n_0+5)} =C(u)\nu^{-(n_0+1)}t^{-(n_0+3)}.
\end{equation}
Therefore, overall we have 
\begin{equation}
\mathbb{E}(\delta(Y)^2) \leq C(u) t^{-(n_0+3)}\nu^{-(n_0+1)},
\end{equation}
as desired.
\end{proof}
With this in hand, we can now easily provide a proof of the main result:
\begin{proof}[Proof of Theorem \ref{main result}]
Recall that $S_\nu(t)f_0(x,y)=\mathbb{E}f_0(X_t^{-1}(x,y))$. Let now $g=g(x,y)\in L^2(\mathbb{T}^2)$ be any function. Changing variables we have
\begin{align}
\int_{\mathbb{T}}P_kS_\nu(t)f_0(y)g(x,y)\dd y \dd x
&=\mathbb{E}\int_{\mathbb{T}^2}g(x,y)e^{ikx}\int_{\mathbb{T}}e^{-ikz}f_0(X_t^{-1}(z,y))\dd z \dd x\dd y\notag\\
&=\mathbb{E}\int_{\mathbb{T}}\int_{\mathbb{T}^2}e^{-ikX_t^1(z,y)}g(x,X^2_t(z,y))f_0(z,y)e^{ikx}\dd z\dd y \dd x.
\end{align}
Note now that $e^{-ikz}g(x,y)= i\partial_z \frac{1}{k}e^{-ikz}g(x,y)$. Hence, by the Malliavin integration by parts formula, it holds 
\begin{equation}
\mathbb{E}(e^{-ikX_t^1(z,y)}g(x,X_t^2(z,y)))=i\frac{1}{k}\mathbb{E}(e^{-ikX_t^1(z,y)}g(x,X_t^2(z,y))\delta(Y(y))),
\end{equation}
where $Y$ is given in \eqref{Y expression}. Applying Cauchy-Schwarz, we thus have
\begin{align}
&\int_{\mathbb{T}^2}|\mathbb{E}(e^{-ikX_t^1(z,y)}g(x,X_t^2(z,y)))|^2\dd z \dd y \notag\\
&\qquad \leq \frac{1}{k^2}\sup_{y \in \mathbb{T}}\mathbb{E}(\delta(Y_t(y))^2)\int_{\mathbb{T}^2}|e^{-ikX_t^1(z,y)}g(x,X_t^2(z,y)))|^2 \dd z \dd y.
\end{align}
Changing variables again, and applying Lemma \ref{Y bound}, we therefore deduce 
\begin{equation}
\int_{\mathbb{T}}P_kS_\nu(t)f_0(y)g(y)\dd y \leq C(u)\|f_0\|_{L^2(\mathbb{T}^2)}\|g\|_{L^2(\mathbb{T})}(k^{-2}t^{-(n_0+3)}\nu^{-(n_0+1)})^\frac{1}{2},
\end{equation}
for any $t\leq \nu^{-1}$. Therefore, by duality it holds 
\begin{equation}
\|P_kS_\nu(t)f_0\|_{L^2(\mathbb{T})}\leq C(u)\|f_0\|_{L^2(\mathbb{T}^2)}(k^{-2}t^{-(n_0+3)}\nu^{-(n_0+1)})^\frac{1}{2}.
\end{equation}
Finally, we want to take $t=(4C(u)^2\nu^{-(n_0+1)}k^{-2})^\frac{1}{n_0+3}$, at which point the $L^2$ norm will have been reduced by a half. But note that our inequality only holds for $t\leq \nu^{-1}$. Therefore, we get the claimed decay rate, as soon as 
\begin{equation}
(4C(u)^2\nu^{-(n_0+1)}k^{-2})^\frac{1}{n_0+3} \leq  \nu^{-1}
\end{equation}
which occurs when 
\begin{equation}
\frac{\nu}{|k|}\leq \frac{1}{2C(u)}
\end{equation}
This concludes the proof
\end{proof}
\subsection{Proof of the bounds on the Malliavin determinant} \label{Malliavin Proofs}
All that remains is to provide a proof of the bounds on the inverse moments of $\det(M_t)(y)$ in Lemma \ref{Important bounds}. 
As in the proof of H\"ormander theorem, the idea will be to estimate the probability of $\{\det(M_t)\leq \epsilon\}$. However, we want our estimates to have some uniformity in $\nu, t$. We will therefore introduce the following notation, which is a slightly modified version of the one in \cite{Hairer,Hairer_Mattingly_2011_}: We say a family of events $A^{\zeta}_\epsilon$ indexed by parameters $\zeta,\epsilon$ is ``almost false", if for any $p$ there exists a constant $C_p$ \emph{independent of $\zeta$} so that $\mathbb{P}(A^\zeta_\epsilon) \leq C_p \epsilon^p$ for all $\zeta$, and for all $\epsilon$ sufficiently small, independent of $\zeta$. Furthermore, we write $A^\zeta_\epsilon \Rightarrow_\epsilon B^\zeta_\epsilon$ (read, $A$ ``almost implies" $B$), to mean $A^\zeta_\epsilon \setminus B^\zeta_\epsilon$ is almost false. Furthermore, given families of random variables $Z_\epsilon^\zeta, W_\epsilon^\zeta$, we write $Z \leq_\epsilon W$ to mean that the event $\{Z_\epsilon^\zeta > W_\epsilon^\zeta\}$ is almost false. In the subsequent sections, the paramater $\zeta$ will typically consist of $\nu\in (0,1]$, $t \in (0,\nu^{-1}]$, as well as possibly some other free parameters which are independent of $\epsilon$.

With these preliminaries out of the way, we state the following lemma which will be fundamental to our discussion: (Note the presence of the additional parameter $v$, which will be essential later)
\begin{lemma}
\label{bound lemma}
Let $v,\gamma, \beta>0$ be arbitrary constants, where $v$ may be random. For any $\alpha <\frac{1}{2}$, there exist positive constants $C,q$ depending only on $\gamma, \beta,\alpha$ so that
\begin{align}
&\int_{0}^t \left(\int_r^tu'(y+\sqrt{\nu}B_s)\dd s-\frac{1}{t}\int_0^t \int_m^t u'(y+\sqrt{\nu}B_s)\dd s\dd m\right)^2\dd r \leq \epsilon v t^{\gamma}\nu^\beta\notag \\
&\quad\Rightarrow_\epsilon \quad \|u'(y+\sqrt{\nu}B_s)\|_{L^\infty_t}\leq C \epsilon^q \max \left\{v^\frac{1}{2}t^\frac{\gamma-3}{2}\nu^\frac{\beta}{2},(v^\alpha t^{\alpha (\gamma-3)} \nu^{\alpha \beta}h(y,\nu,t)^3 \nu^\frac{3}{2}t^\frac{3}{2})^\frac{1}{3+2\alpha},\right.\notag \\
&\hspace{7cm}\left. (v^\frac{\alpha}{2}t^\frac{\alpha (\gamma-3)+1}{2}\nu^\frac{\alpha \beta +1}{2} h(y,\nu,t))^\frac{1}{1+\alpha}\right\}.
\end{align}
where we have set 
$$
h(y,\nu,t)=\sup_{s,z \in [0,t],\tau \in [0,1]}|u''(y+\sqrt{\nu}(\tau B_s+(1-\tau)B_z))|.
$$
\end{lemma}
We will postpone the proof of this to end of the section.
With this lemma under our belt, we now prove the following, which effectively states that the assumption that $\det(M)$ is small already implies that our process is localised around one of the critical points of $u'$. In particular, this will allow us to Taylor expand around these critical points, and hence to replace $\|u'(y+\sqrt{\kappa}B_s)\|_{L^\infty_t}$ by a suitable power of $\|y+\sqrt{\kappa}B_s\|_{L^\infty_t}$.
\begin{lemma}
\label{small norm}
Fix any number $\eta$ sufficiently small, and let $\Tilde{v}$ be any deterministic positive number. Denote by $A$ the set of zeros of $u'(y)$, and let $\omega>0$ be any positive real number. Then, it holds that for $\epsilon$ small enough depending only on $\omega ,u,\Tilde{v},\eta$, and for any $t \leq \nu^{-1}$, we have
\begin{align*}
&\int_{0}^t \left(\int_r^tu'(y+\sqrt{\nu}B_s)\dd s-\frac{1}{t}\int_0^t \int_m^t u'(y+\sqrt{\nu}B_s)\dd s\dd m\right)^2\dd r \leq \epsilon \Tilde{v} t^{\omega+3}\nu^\omega \\
&\quad \Rightarrow_\epsilon \quad\sup_{s\leq t}\mathrm{dist}(y+\sqrt{\nu}B_s,A) \leq \eta.
\end{align*}
In particular, suppose that $u$ has a critical point of order $n$ at $y_0$. Fix $\eta(y_0)>0$ sufficiently small, so that for $|z-y_0|\leq \eta(y_0)$, there holds 
\begin{equation}
\label{Taylor expansion}
c_1 |z-y_0|^n \leq |u'(z)| \leq c_2 |z-y_0|^n, \qquad 
c_3|z-y_0|^{n-1} \leq |u''(z)| \leq c_4|z-y_0|^{n-1}.
\end{equation}
Then, if $|y-y_0|<\eta$, we have the almost implication
\begin{align*}
&\int_{0}^t \left(\int_r^tu'(y+\sqrt{\nu}B_s)\dd s-\frac{1}{t}\int_0^t \int_m^t u'(y+\sqrt{\nu}B_s)\dd s\dd m\right)^2\dd r \leq \epsilon \Tilde{v}t^{\omega+3}\nu^\omega   \\
&\quad \Rightarrow_\epsilon \quad h(y,\nu,t) \leq c(u)\|y-y_0+\sqrt{\nu}B_s\|_{L^\infty_t}^{n-1}.
\end{align*}
\end{lemma}
\begin{proof}
Set $y_0=0$ for convenience. Note that on the event $\eta <\sup_{s\leq t}\mathrm{dist}(y+\sqrt{\nu}B_s,A)$, we have that $\|u'(y+\sqrt{\nu}B_t)\|_{L^\infty_t} \geq \sigma$ for $\sigma=\inf_{\mathrm{dist}(x,A)\geq \eta}|u'(x)|>0$. Therefore we simply need to check the three possible cases from Lemma \ref{bound lemma}. Using the trivial bound $h(y,\nu,t)\leq \|u''\|_{L^\infty(\mathbb{T})}$, we see that in any of the cases from Lemma \ref{bound lemma}, we have the almost inequality (assuming $t\leq \nu^{-1}$)
\begin{equation}
\sigma \leq_{\epsilon} D(C,\|u''\|_{L^\infty(\mathbb{T})})\Tilde{v}^\frac{\omega}{2}\epsilon^q t^\frac{\omega}{2}\nu^\frac{\omega}{2} \leq D(C,\|u''\|_{L^\infty(\mathbb{T})}) \Tilde{v}^\frac{\omega}{2} \epsilon^{q},
\end{equation}
where $D$ is some constant depending only on $C$ and $\|u''\|_{L^\infty(\mathbb{T})}$. Then, for all $\epsilon <(\frac{1}{2}\Tilde{v}^\frac{\omega}{2}\sigma D )^\frac{1}{q}$, this bound cannot hold, completing the proof of the first claim. To prove the second part, we simply note that we have 
\begin{equation}
h(y,\nu,t)\leq_\epsilon c_4 \sup_{s,z \in [0,t],\tau \in [0,1]}|y-y_0+\sqrt{\nu}(\tau B_z+(1-\tau)B_s)|^{n-1} \leq c_4\|y-y_0+\sqrt{\nu}B_s\|_{L^\infty_t}^{n-1},
\end{equation}
where the last inequality follows by convexity.
\end{proof}
We now have all the tools to prove  Lemma \ref{Important bounds}. We begin by considering the first bound:
\begin{lemma}
\label{initial det estimate}
Let the maximal order of critical points of $u$ be $n_0$. Then, for any $p\geq 1$, there exists a constant $\Tilde{K}(u,p)>0$ independent of $\nu, t\leq \nu^{-1}$ so that 
\begin{equation}
\sup_{y \in \mathbb{T}}\mathbb{E}(\det(M(y))^{-p}) \leq \Tilde{K}(u,p)t^{-p(n_0+3)}\nu^{-n_0p}.
\end{equation}
\end{lemma}
\begin{proof}
By standard arguments, it suffices to show that $\det(M_t)<\epsilon t^{n_0+3}\nu^{n_0} \Rightarrow_\epsilon \emptyset$. In view of Lemma \ref{small norm}, letting $\eta$ be so that the bounds \eqref{Taylor expansion} hold for all the critical points, we see that if $\mathrm{dist}(y,A)> \eta$, then indeed $\det(M_t)<\epsilon t^{n_0+3}\nu^{n_0} \Rightarrow_\epsilon \emptyset $. Hence, pick some $y_0 \in A$, say that the critical point at $y_0$ is of order $n$. For convenience sake, assume $y_0=0$. By Lemma \ref{small norm}, we see that in particular we have $h(y,\nu,t)\leq_\epsilon c(u) \|y+\sqrt{\nu}B_s\|^{n-1}_{L^
\infty}$. Similarly, we have 
\begin{equation}
\|y+\sqrt{\nu}B_s\|_{L^\infty_t}^n \leq_\epsilon c_1 \|u'(y+\sqrt{\nu}B_s)\|_{L^\infty_t}.
\end{equation}
Let us examine the upper bounds in Lemma \ref{bound lemma} again. In particular, the second bound becomes
\begin{equation}
\|y+\sqrt{\nu}B_s\|_{L^\infty_t}^{n-\frac{3(n-1)}{3+2\alpha}} \leq_\epsilon \Tilde{C}(u,n) \epsilon^q (t^{\alpha n+\frac{3}{2}}\nu^{\alpha n+\frac{3}{2}})^\frac{1}{3+2\alpha},
\end{equation}
and rearranging yields 
\begin{equation}
\|y+\sqrt{\nu}B_s\|_{L^\infty_t} \leq_\epsilon \Tilde{C}(u,n,\alpha)\epsilon^l t^\frac{1}{2}\nu^\frac{1}{2},
\end{equation}
for some $l>0$. Similarly, the third bound also reduces to this, and the first one also does trivially. Therefore, we just need to show 
\begin{equation}
\|y+\sqrt{\nu}B_s\|_{L^\infty_t} \leq \Tilde{C}(u,n,\alpha)\epsilon^l t^\frac{1}{2}\nu^\frac{1}{2}\quad \Rightarrow_\epsilon \quad\emptyset .
\end{equation}
We now split into two cases. Firstly, assume that $|y|\leq \frac{1}{2}\|\sqrt{\nu}B_s\|_{L^\infty_t}$. Then by the triangle inequality, we have 
\begin{equation}
\frac{1}{2}\|\sqrt{\nu}B_s\|_{L^\infty_t}\leq \|y+\sqrt{\nu}B_s\|_{L^\infty_t} \leq_\epsilon \Tilde{C}(u,n,\alpha)\epsilon^l t^\frac{1}{2}\nu^\frac{1}{2}.
\end{equation}
Next assume that $|y|\geq \frac{1}{2}\|\sqrt{\nu}B_s\|_{L^\infty_t}$. Then, we have that 
\begin{equation}
\|\sqrt{\nu}B_s\|_{L^\infty_t} \leq 2|y| \leq 2\|y+\sqrt{\nu}B_s\|_{L^\infty_t} \leq_\epsilon 2\Tilde{C}(u,n,\alpha) \epsilon^l t^\frac{1}{2}\nu^\frac{1}{2}.
\end{equation}
Hence, in either case we have 
\begin{equation}
\|\sqrt{\nu}B_s\|_{L^\infty_t}\leq_\epsilon 2\Tilde{C}(u,n,\alpha) \epsilon^l t^\frac{1}{2}\nu^\frac{1}{2}.
\end{equation}
Using Brownian scaling, this can easily be seen to be almost false. Therefore, there exists some $\epsilon_0(y_0)$ depending only on $u,\alpha$, so that for any $p>0$, there exists some $C_p(y_0)$ depending only on $u,p$, so that for any $y \in B_{\eta}(y_0)$, there holds 
\begin{equation}
\mathbb{P}(\det(M_t)(y)\leq \epsilon t^{n+3}\nu^n) \leq C_p(y_0)\epsilon^p,
\end{equation}
Furthermore, one can find analogous $\epsilon_0(y_i), C_p(y_i)$ for the other critical points of $u$, as well as trivially in the case where $\mathrm{dist}(y,A)\geq \eta$. Hence, for any $p\geq 1$, there exists a constant $K(u,p)$, so that for any $y$ in $B_{\eta}(y_i)$, where $y_i$ is a critical point of order $m$ say, it holds $\mathbb{E}(\det(M_t)(y)^{-p}) \leq K(u,p) t^{-p(m+3)}\nu^{-pm}$, (and for $\mathrm{dist}(y,A)\geq \eta$ it holds with $m=n_0$). Therefore, using furthermore that $t \leq \nu^{-1}$, we may in fact conclude that 
\begin{equation}
\sup_y \mathbb{E}(\det(M_t)(y)^{-p})\leq \Tilde{K}(u,p)t^{-p(n_0+3)}\nu^{-pn_0}
\end{equation}
completing the proof. 
\end{proof}
With this in hand, the second bound from Lemma \ref{Important bounds} follows similarly.
\begin{lemma}
Let $v =t\|u''(y+\sqrt{\nu}B_s)\|_{L^\infty_t}$, with $y$ close to a critical point of order $n$. Then there holds 
\begin{equation}
\det(M_t) \leq \epsilon v t^\frac{n+5}{2}\nu^\frac{n+1}{2}\quad \Rightarrow_\epsilon\quad \emptyset.
\end{equation}
In particular, it holds that for any $p\geq 1$, there exists $C(p)>0$ so that for $t \leq \nu^{-1}$
\begin{equation}
\mathbb{E}((\frac{t\|u''(y+\sqrt{\nu}B_s)\|_{L^\infty_t}}{\det(M_t)})^p)\leq C_p t^{-p\frac{n_0+5}{2}}\nu^{-p\frac{n_0+1}{2}}.
\end{equation}
\end{lemma}
\begin{proof}
Noting that $v\leq t\Tilde{v}$, where $\Tilde{v} = \|u''\|_{L^\infty(\mathbb{T})}$, we can apply Lemma \ref{small norm}, since the upper bound is now of the form $\epsilon^q \Tilde{v} t^{\omega+3}\nu^{\omega}$, with $\omega=\frac{n+1}{2}$. Next, we argue as in Lemma \ref{initial det estimate}, and observe that applying Lemma \ref{bound lemma} with $v=t\|u''(y+\sqrt{\kappa}B_s)\|_{L^\infty_t}$, together with the facts that close to a critical point, $\|u'(y+\sqrt{\nu}B_s)\|_{L^\infty_t} \geq c_1\|y+\sqrt{\nu}B_s\|_{L^\infty_t}^n$, and similarly $\|u''(y+\sqrt{\nu}B_s)\|_{L^\infty_t} \leq c_4\|y+\sqrt{\nu}B_s\|_{L^\infty_t}^{n-1}$, we get once again
\begin{equation}
\|y+\sqrt{\nu}B_s\|_{L^\infty_t} \leq_\epsilon \Tilde{C}(u,n,\alpha) \epsilon^q t^\frac{1}{2}\nu^\frac{1}{2}
\end{equation}
so we may conclude in exactly the same way as in Lemma \ref{initial det estimate}.
\end{proof}
All that remains now is to provide a proof of Lemma \ref{bound lemma}. To do so, we recall first the following well-known interpolation lemma (see e.g. \cite{Hairer_Mattingly_2011_})
\begin{lemma}
\label{interpolation lemma}
Let $f:[0,t] \to \mathbb{R}$ be $C^{1,\alpha}$, for some $\alpha \in (0,1]$. Then it holds 
\begin{equation}
\|\partial_t f\|_{L^\infty_t} \leq 4\|f\|_{L^\infty_t} \max\{t^{-1},\|f\|_{L^\infty_t}^{-\frac{1}{1+\alpha}}\|\partial_t f\|_{C^\alpha_t}^\frac{1}{1+\alpha}\}
\end{equation}
where $\|\partial_tf\|_{C^\alpha_t}$ denotes the best $\alpha$-H\"older constant of $\partial_t f$.
\end{lemma}
With this in hand, we can then provide the proof of Lemma \ref{bound lemma}
\begin{proof}[Proof of Lemma \ref{bound lemma}]
Let $f(z)=\int_0^z(\int_{r}^tu'(y+\sqrt{\nu}B_s) \dd s -\frac{1}{t}\int_{0}^t\int_m^t u'(y+\sqrt{\nu}B_s)\dd s\dd m)^2\dd r$. This is an increasing function in $z$, and so 
\begin{equation}
\|f\|_{L^\infty_t}=f(t)
\end{equation}
Therefore, consider the event 
\begin{equation}
\label{almost false event}
\|f\|_{L^\infty_t} \leq \epsilon t^{\gamma} \nu ^\beta v
\end{equation}
In the following, we will apply Lemma \ref{interpolation lemma} repeatedly. As we are interested in ensuring that $\epsilon$ can be chosen small independently of $t, \nu$, we will techincally always need to split up into the cases $t^{-1} \leq \|f\|_{L^\infty_t}^{-\frac{1}{1+\alpha}}\|\partial_t f\|_{C^\alpha_t}^\frac{1}{1+\alpha}$ and the converse (this is why the three different upper bounds appear in the statement of the result).  As this would be slightly arduous at times, and the situation where $t^{-1} \leq \|f\|_{L^\infty_t}^{-\frac{1}{1+\alpha}}\|\partial_t f\|_{C^\alpha_t}^\frac{1}{1+\alpha}$ will tend to be the more interesting one, we omit the other cases. However, they follow in exactly the same way.

Firstly, by Lemma \ref{interpolation lemma} we see that \eqref{almost false event} implies 
\begin{align}
&\left\|\left(\int_{r}^tu'(y+\sqrt{\nu}B_s) \dd s -\frac{1}{t}\int_{0}^t\int_m^t u'(y+\sqrt{\nu}B_s)\dd s\dd m\right)^2\right\|_{L^\infty_t} \notag\\
&\leq 4 \sqrt{2}\epsilon^\frac{1}{2}v^\frac{1}{2}t^\frac{\gamma}{2}\nu^\frac{\beta}{2}\left\|\left(\int_{r}^tu'(y+\sqrt{\nu}B_s) \dd s -\frac{1}{t}\int_{0}^t\int_m^t u'(y+\sqrt{\nu}B_s)\dd s\dd m\right)\right\|_{L^\infty_t}^\frac{1}{2} \notag\\
&\qquad \times\left\|u'(y+\sqrt{\nu}B_s)\right\|^\frac{1}{2}_{L_t^\infty}.
\end{align}
where we have used the relation
\begin{equation}
\|h^2\|_{\mathrm{Lip}_t} \leq 2\|h\|_{L^\infty_t}\|h\|_{\mathrm{Lip}_t}.
\end{equation}
where $\mathrm{Lip}_t$ denotes the best Lipschitz constant on $[0,t]$.
Rearranging, we see that 
\begin{align*}
&\left\|\left(\int_{r}^tu'(y+\sqrt{\nu}B_s) \dd s -\frac{1}{t}\int_{0}^t\int_m^t u'(y+\sqrt{\nu}B_s\dd s)\dd m\right)\right\|_{L^\infty_t}\\
&\qquad\qquad\leq (32)^\frac{1}{3}\epsilon^\frac{1}{3}v^\frac{1}{3}t^\frac{\gamma}{3}\nu^\frac{\beta}{3}\|u'(y+\sqrt{\nu}B_s)\|_{L^\infty_t}^\frac{1}{3}.
\end{align*}
Applying Lemma \ref{interpolation lemma} again, this time to 
\begin{equation}
r \mapsto \int_{r}^tu'(y+\sqrt{\nu}B_s) \dd s -\frac{1}{t}\int_{0}^t\int_m^t u'(y+\sqrt{\nu}B_s)\dd s\dd m
\end{equation}
we see that this implies for any $\alpha \in (0,\frac{1}{2})$
\begin{equation}
\|u'(y+\sqrt{\nu}B_s)\|_{L^\infty_t} \leq 4((32)^\frac{1}{3}\epsilon^\frac{1}{3}v^\frac{1}{3}t^\frac{\gamma}{3}\nu^\frac{\beta}{3}\|u'(y+\sqrt{\nu}B_s)\|_{L^\infty_t}^\frac{1}{3})^\frac{\alpha}{1+\alpha}(\|u'(y+\sqrt{\nu}B_s\|_{C^\alpha_t})^\frac{1}{1+\alpha}.
\end{equation}
Now, note that using the usual Brownian scaling, it holds that $\|\sqrt{\nu}B_s\|_{C^\alpha_t} \leq_{\epsilon} \epsilon^{-\delta}\nu^\frac{1}{2}t^{\frac{1}{2}-\alpha}$. Therefore, we have
\begin{equation}
\|u'(y+\sqrt{\nu}B_s)\|_{C^\alpha_t} \leq h(y,\nu,t) \|\sqrt{\nu}B_s\|_{C^\alpha_t}\leq_\epsilon \epsilon^{-\delta}h(y,\nu,t)\nu^{\frac{1}{2}}t^{\frac{1}{2}-\alpha},
\end{equation}
so that 
\begin{equation}
\|u'(y+\sqrt{\nu}B_s)\|_{L^\infty_t} \leq_\epsilon (32^\frac{\alpha}{3}\epsilon^{\alpha-3\delta}v^{\alpha}t^{\alpha(\gamma-3)}\nu^{\alpha \beta}h(y,\nu,t)^3\nu^{\frac{3}{2}}t^{\frac{3}{2}})^\frac{1}{3+2\alpha}
\end{equation}
completing the proof of the first case. The other cases follow similarly.
\end{proof}

\addtocontents{toc}{\protect\setcounter{tocdepth}{0}}

\section*{Acknowledgements}
The author thanks Michele Coti Zelati and Martin Hairer for helpful comments and inspiring discussions. The research of DV was funded by the Imperial College President's PhD Scholarships.

\addtocontents{toc}{\protect\setcounter{tocdepth}{1}}

\bibliographystyle{alpha}
\bibliography{bib.bib}

\end{document}